\documentclass{amsart}

\usepackage{amsmath,amssymb,amsthm}

\newtheorem{theorem}{Theorem}
\newtheorem{lemma}{Lemma}

\begin{document}


%
%

\title[Multiplicative functions commutable with sums of squares]{Multiplicative functions commutable\\ with sums of squares}
\author{Poo-Sung Park}
\address{Department of Mathematics Education, Kyungnam University, Changwon, Republic of Korea}
\email{pspark@kyungnam.ac.kr}

\thanks{This work was supported by Kyungnam University Foundation Grant, 2016.}


\maketitle

%
\begin{abstract}
Let $k$ be an integer greater than or equal $3$. We show that if a multiplicative function $f$ satisfies
\[
f(x_1^2 + x_2^2 + \dots + x_k^2) = f(x_1)^2 + f(x_2)^2 + \dots + f(x_k)^2
\]
for all positive integers $x_i$, then $f$ is the identity function.
\end{abstract}


\section{Introduction}

In 1992, Claudia Spiro \cite{Spiro} proved that if a multiplicative function $f:\mathbb{N} \to \mathbb{C}$ satisfies $f(p+q) = f(p)+f(q)$ for all primes $p$ and $q$, then $f(n)=n$ under the condition $f(p_0) \ne 0$ for some prime $p_0$. The term \emph{multiplicative function} means $f(1)=1$ and $f(mn) = f(m)f(n)$ for all coprime integers $m$ and $n$. 

It has influenced lots of mathematicians to develop her problem into diverse variations. Fang \cite{Fang} showed a similar result for $f(p+q+r) = f(p)+f(q)+f(r)$ and Dubickas and \v{S}arka \cite{D-S} extended those results to $f(p_1 + p_2 + \dots + p_k) = f(p_1) + f(p_2) + \dots + f(p_k)$. Chung \cite{Chung} classified multiplicative functions satisfying the equation $f(m^2+n^2) = f(m^2)+f(n^2)$  for all $m,n \in \mathbb{N}$. Recently, Ba\v{s}i\'c \cite{Basic} characterized all arithmetic functions such that $f(m^2+n^2) = f(m)^2 + f(n)^2$, which is slightly different from Chung's condition. Many related problems are found in \cite{C-C, CFYZ, Chung, C-P, DK-K-P, I-P, Phong2004, Phong2006}. We would call such problems \emph{Spiro problems}.

In this article, we add new results to the list of Spiro problems by proving the following: the condition 
\[
f(x_1^2 + x_2^2 + \dots + x_k^2) = f(x_1)^2 + f(x_2)^2 + \dots + f(x_k)^2
\]
for all positive integers $x_i$ forces a multiplicative function $f$ to be the identity function for $k \ge 3$. 


\section{Results}

For the consistency, we mention the sum of two squares. 

\begin{theorem}\label{thm:sum_of_2_squares}
If a multiplicative function $f$ satisfies
\[
f(x^2+y^2) = f(x)^2+f(y)^2
\]
for all positive integers $x$ and $y$, then $f(n^2)=f(n)^2=n^2$.
\end{theorem}

\begin{proof}
This is a consequence of Ba\v{s}i\'c's characterization of arithmetic functions satisfying the abovementioned condition \cite{Basic}.
\end{proof}

To prove the main theorem, we need lemmas about sums of three or more nonvanishing squares. See \cite[Chapter 6]{Grosswald}, \cite{Hurwitz}, \cite{Dubouis}. 

\begin{lemma}[Hurwitz]\label{Hurwitz}
The only squares that are not sums of three nonvanishing squares are the integers $4^s$ and $25\cdot4^s$ with $s \ge 0$.
\end{lemma}

\begin{lemma}[Dubouis]\label{Dubouis}
Every integer $n$ can be represented by the sum of $k$ nonvanishing squares except
\[
n = \begin{cases}
1, 3, 5, 9, 11, 17, 29, 41, 2\cdot4^m, 6\cdot4^m, 14\cdot4^m ~(m \ge 0) & \text{if } k = 4, \\
1, 2, 3, 4, 6, 7, 9, 10, 12, 15, 18, 33 & \text{if } k = 5, \\
1, 2, \dots, k-1, k+1, k+2, k+4, k+5, k+7, k+10, k+13 & \text{if } k > 5.
\end{cases}
\]
\end{lemma}

Since sums of $k$ nonvanishing squares represent almost all positive integers for $k \ge 5$, we separate the problem into three cases: $k=3$, $k=4$, and $k \ge 5$. 

\begin{theorem}\label{thm:sum_of_3_squares}
If a multiplicative function $f$ satisfies
\[
f(x^2+y^2+z^2) = f(x)^2+f(y)^2+f(z)^2
\]
for all positive integers $x$, $y$ and $z$, then $f$ is the identity function.
\end{theorem}

\begin{proof}
First, we compute $f(n)$ for $n \le 15$ and $n=25$.

Since $f$ is multiplicative, $f(1)=1$. Clearly, $f(3) = f(1^2+1^2+1^2) = 3$. To find $f(2)$ let us consider
\[
f(6) = f(1^2+1^2+2^2) = f(1)^2 + f(1)^2 + f(2)^2 = 2 + f(2)^2.
\]
Since $f(6) = f(2)f(3) = 3f(2)$, we find that $f(2) = 1$ or $f(2)=2$ from the equation $2+f(2)^2 = 3f(2)$. To determine which value $f(2)$ is we need more relations.

From $f(12) = f(2^2+2^2+2^2) = 3f(2)^2 = f(3)f(4) = 3f(4)$ we have $f(4) = f(2)^2$.

Furthermore, from $f(14) = f(1^2+2^2+3^2) = 10+f(2)^2 = f(2)f(7)$ we observe that $f(7) = f(2)+\frac{10}{f(2)}$.

To find another relation between $f(2)$ and $f(7)$, consider 
\[
f(21) = f(1^2+2^2+4^2) = 1+f(2)^2 + f(2)^4 = f(3)f(7) = 3f(7).
\]
From the equation
\[
1+f(2)+f(2)^2 = 3 \left( f(2) + \frac{10}{f(2)} \right)
\]
we can conclude that $f(2)=2$ and thus $f(4) = 4$, $f(6)=6$, $f(7)=7$, $f(12)=12$, and $f(14) = 14$.

To find $f(5)$ we examine $f(27)$ and $f(30)$. Since
\begin{align*}
f(27) 
&= f(1^2+1^2+5^2) = 2+f(5)^2 \\
&= f(3^2+3^2+3^2) = 3f(3)^2 = 27,
\end{align*}
we have $f(5) = \pm5$. On the other hand,
\begin{align*}
f(30) 
&= f(1^2+2^2+5^2) = 5+f(5)^2 \\
&= f(5)f(6) = 6f(5).
\end{align*}
Thus, we have that $f(5)=5$.

Since 
\begin{align*}
f(50) 
&= f(3^2+4^2+5^2) = 50 \\
&= f(2)f(25) = 2f(25),
\end{align*}
we have $f(25)=25$ and $f(50)=50$.

Now, to find $f(8)$ consider $f(24)$. Since
\begin{align*}
f(24) 
&= f(2^2+2^2+4^2) = 24 \\
&= f(3)f(8) = 3f(8),
\end{align*}
we have that $f(8) = 8$. 

Similarly, it follows from
\begin{align*}
f(26) 
&= f(1^2+3^2+4^2) = 26 \\
&= f(2)f(13) = 2f(13)
\end{align*}
that $f(13)=13$.

Other $f(n)$'s can be calculated as follows:
\begin{align*}
f(9) &= f(1^2+2^2+2^2) = 9, 
&f(10) &= f(2)f(5) = 10, \\
f(11) &= f(1^2+1^2+3^2) = 11, 
& f(15) &= f(3)f(5) = 15.
\end{align*}

Now consider $n > 15$. We use induction. Assume that $f(k)=k$ for all positive integers $k < n$.

If $n = a^2+b^2+c^2$ with $a,b,c \ge 1$, then 
\[
f(n) = f(a^2+b^2+c^2) = f(a)^2+f(b)^2+f(c)^2 = a^2+b^2+c^2 = n
\]
since $f(a)=a$, $f(b)=b$, $f(c)=c$ by induction hypothesis.

If $n$ cannot be represented as a sum of three nonvanishing squares, this means two cases:
\begin{enumerate}
\item if $n = a^2 + b^2 + c^2$ for some $a,b,c \in \mathbb{Z}$, then at least one of $a,b,c$ should vanish.
\item $n \ne a^2 + b^2 + c^2$ for any $a,b,c \in \mathbb{Z}$.
\end{enumerate}

Consider case (1): Suppose $n = a^2 + b^2$ with $3 \le a \le b$. Note that 
\[
5a < 2a^2 \le a^2 + b^2 = n, \quad
3b < 4b \le b^2 + 4 < b^2 + a^2 = n.
\]
If $5 \nmid n$, then
\begin{align*}
f(25n) 
&= f(25a^2 + 25b^2) = f\!\left((5a)^2 + (3b)^2 + (4b)^2\right) \\
&= f(5a)^2 + f(3b)^2 + f(4b)^2 \\
&= (5a)^2 + (3b)^2 + (4b)^2 = 25n
\end{align*}
by the induction hypothesis. Since $n$ is assumed to be indivisible by $5$, $f(25n) = f(25)f(n) = 25f(n)$ and $f(n) = n$.

If $n = a^2 + b^2$ with $1 \le a \le 2$ and $4 \le b$, then
\[
5a \le 10 < n, \quad
3b < 4b \le b^2 < a^2 + b^2 = n
\]
and thus $f(n) = n$ by the same reasoning.

It is already checked that $f(n)=n$ for $a \le 2$ and $b \le 3$ in the previous step. 

If $n = a^2 + b^2$ is divisible by $5$, let $n = 5^s k$ with $s \ge 1$ and $5 \nmid k$. We have shown that $f(5) = 5$ and $f(25)=25$. If $s>2$ is even, $5^s$ can be represented as a sum of three nonvanishing squares by Lemma \ref{Hurwitz}. So $f(5^s) = 5^s$.

If $s \ge 3$ is odd, then $s-3$ is even and 
\[
5^s = 5^{s-3} \cdot 125 = 5^{s-3}(3^2 + 4^2 + 10^2).
\]
Thus, $f(5^s) = 5^s$.

We can conclude that $f(n) = f(5^s k) = f(5^s) f(k) = 5^s k = n$.

Now we should investigate the case of $n = a^2$. By Lemma \ref{Hurwitz} it suffices to check whether $f(4^s) = 4^s$. Note that
\begin{align*}
f(3 \cdot 4^s) 
&= f\!\left((2^s)^2 + (2^s)^2 + (2^s)^2\right) = f(2^s)^2 + f(2^s)^2 + f(2^s)^2 = 3 \cdot 4^s \\
&= f(3)f(4^s) = 3f(4^s).
\end{align*}
Thus $f(4^s) = 4^s$ and we can conclude that $f(n)=n$ for the case (1).


Now consider the case (2): It is well known that $n$ cannot be represented as a sum of three squres if and only if $n$ is of the form $4^s(8t+7)$ with $s,t \ge 0$. We may assume that $n = 8t+7$. Then, since $2(8t+7) = 8(2t+1)+6$ can be represented as a sum of three squares of integers, $f\!\left(2(8t+7)\right) = 2(8t+7)$ and thus $f(8t+7) = 8t+7$ by canceling $f(2)=2$.

From the case (1) and (2) we can conclude that $f(n)=n$.
\end{proof}

By interweaving the method in the above proof with Fermat's Two-Square Theorem, we can also prove Theorem \ref{thm:sum_of_2_squares} in the similar way.

\begin{theorem}\label{thm:sum_of_4_squares}
If a multiplicative function $f$ satisfies
\[
f(x_1^2 + x_2^2 + x_3^2 + x_4^2) = f(x_1)^2 + f(x_2)^2 + f(x_3)^2 + f(x_4)^2
\]
for all positive integers $x_i$, then $f$ is the identity function.
\end{theorem}

\begin{proof}
By Lemma \ref{Dubouis} every positive integer can be represented as sums of four nonvanishing squares except for
\[
1, 3, 5, 9, 11, 17, 29, 41, 2\cdot4^m, 6\cdot4^m, 14\cdot4^m ~(m \ge 0).
\]
Note that $f(1) = 1$ and $f(4) = 4$.

Since 
\begin{align*}
f(12) 
&= f(3)f(4) = 4f(3) \\
&= f(1)^2 + f(1)^2 + f(1)^2 + f(3)^2 = 3+f(3)^2,
\end{align*}
we obtain that $f(3) = 1$ or $f(3) = 3$.

Let us consider $f(5)$. Since
\begin{align*}
f(20) 
&= f(4)f(5) = 4f(5) \\
&= f(1)^2 + f(1)^2 + f(3)^2 + f(3)^2 = 2+2f(3)^2,
\end{align*}
$f(5) = 1$ or $f(5) = 5$ according to the value of $f(3)$.  

Similarly, from
\begin{align*}
f(28) 
&= f(4)f(7) = 4f(7) \\
&= f(1)^2 + f(3)^2 + f(3)^2 + f(3)^2 = 1+3f(3)^2
\end{align*}
we have that $f(7) = 1$ or $f(7) = 7$ according to the value of $f(3)$.

Note that
\begin{align*}
f(35) 
&= f(5)f(7) \\
&= f(1)^2 + f(3)^2 + f(3)^2 + f(4)^2 = 17+2f(3)^2.
\end{align*}
Hence we conclude that $f(3) = 3$, $f(5) = 5$, and $f(7) = 7$. Then $f(2) = 2$ also follows from
\begin{align*}
f(10) 
&= f(2)f(5) \\
&= f(1)^2 + f(1)^2 + f(2)^2 + f(2)^2 = 2+2f(2)^2
\end{align*}
and
\[
f(7) = f(1)^2 + f(1)^2 + f(1)^2 + f(2)^2 = 3+f(2)^2.
\]

If we consider
\begin{align*}
f(18) 
&= f(2)f(9) = 2f(9) \\
&= f(1)^2 + f(2)^2 + f(2)^2 + f(3)^2 = 18,
\end{align*}
we have that $f(9) = 9$. The similar observation gives us
\[
f(11) = 11, \quad f(17) = 17, \quad f(29) = 29, \quad f(41) = 41.
\]

Now suppose that $f(2\cdot4^r) = 2\cdot4^r$ for $r < m$. Note that
\[
f(6\cdot4^r) = f(3)f(2\cdot4^r) = 6\cdot2^r
\quad\text{and}\quad
f(14\cdot4^r) = f(7)f(2\cdot4^r) = 14\cdot2^r.
\]
Thus $f(n)=n$ for $n < 2\cdot4^m$ by Lemma \ref{Dubouis}. Since $5 \cdot 2 \cdot 4^m$ can be represented as a sum of squares of four positive integers less than $2\cdot4^m$, we can conclude that $f(5\cdot2\cdot4^m) = 5\cdot2\cdot4^m$. On the other hand,
\[
f(5\cdot2\cdot4^m) 
= f(5)f(2\cdot4^m) = 5 f(2\cdot4^m)
\]
and thus $f(2\cdot4^m) = 2\cdot4^m$.

Now consider $6\cdot4^m$. Since $f(6 \cdot 4^m) = f(3)f(2\cdot4^m)$, we obtain that $f(6\cdot4^m) = 6\cdot4^m$. Similarly, $f(14\cdot4^m) = f(7)f(2\cdot4^m) = 14 \cdot 4^m$.
\end{proof}

\begin{theorem}\label{thm:sum_of_5_or_more_squares}
Let $k$ be an integer $\ge 5$. If a multiplicative function $f$ satisfies
\[
f(x_1^2+x_2^2+ \dots + x_k^2) = f(x_1)^2+f(x_2)^2+ \dots + f(x_k)^2
\]
for all positive integers $x_i$, then $f$ is the identity function.
\end{theorem}

\begin{proof}
After we prove for $k=5,6,7$ in the way similar to those above, the case of $k \ge 8$ will be proved finally.

\noindent i) Assume that $k=5$. By Lemma \ref{Dubouis} every positive integer can be represented by the sum of five nonvanishing squares except for
\[
1, 2, 3, 4, 6, 7, 9, 10, 12, 15, 18, 33.
\]
We have $f(1)=1$ and $f(5)=5$.

From
\begin{align*}
f(20)	 
&= f(4)f(5) = 5f(4) \\
&= f(1)^2 + f(1)^2 + f(1)^2 + f(1)^2 + f(4)^2 = 4+f(4)^2 \\
&= f(2)^2 + f(2)^2 + f(2)^2 + f(2)^2 + f(2)^2 = 5f(2)^2
\end{align*}
we obtain $f(4) = 1$ or $f(4)=4$. Also, $f(2) = \pm1$ or $f(2) = \pm2$ according to the value of $f(4)$.

From
\begin{align*}
f(29)
&= f(1)^2 + f(1)^2 + f(1)^2 + f(1)^2 + f(5)^2 = 29 \\
&= f(1)^2 + f(1)^2 + f(3)^2 + f(3)^2 + f(3)^2 = 2+3f(3)^2\\
&= f(1)^2 + f(2)^2 + f(2)^2 + f(2)^2 + f(4)^2 = 1+3f(2)^2+f(4)^2
\end{align*}
we obtain that $f(29) = 29$, $f(2) = \pm2$, and $f(3) = \pm3$. Also, this fixes $f(4)=4$, $f(8)=8$ , and $f(11) = 11$.

Then, it follows from 
\begin{align*}
f(2\cdot11) 
&= f(2)f(11) = 11f(2) \\
&= f(1)^2+f(2)^2+f(2)^2+f(2)^2+f(3)^2 = 22
\intertext{and}
f(3\cdot8) 
&= f(3)f(8) = 8f(2) \\
&= f(1)^2+f(1)^2+f(2)^2+f(2)^2+f(4)^2 = 24
\end{align*}
that $f(2)=2$, $f(3)=3$, and $f(6)=6$.

Now, we use induction to show that $f(n) = n $ for $n \ge 7$. Assume that $f(m) = m$ for $m < n$. Then $n(n-1) \ge 42$ and it can be represented as a sum of squares of five positive integers less than $n$. Thus
\[
f\!\left(n(n-1)\right) = n(n-1).
\]
Since $n$ and $n-1$ are relatively prime,
\[
f\!\left(n(n-1)\right) = f(n)f(n-1) = f(n)(n-1)
\]
by induction hypothesis and thus $f(n)=n$.

\bigskip

\noindent ii) Assume that $k = 6$. By Lemma \ref{Dubouis} every positive integer can be represented by the sum of six nonvanishing squares except for
\[
1, 2, 3, 4, 5, 7, 8, 10, 11, 13, 16, 19.
\]
The proof is similar to the previous case. Since $n(n-1) > 19$ for $n \ge 6$, it suffices to show that $f(n) = n$ for $n \le 5$. Clearly, $f(1) = 1$ and $f(6)=6$. From the equalities
\begin{align*}
f(30) 
&= f(5)f(6) = 6 f(5) \\
&= f(1^2+1^2+1^2+1^2+1^2+5^2) = 5+f(5)^2 \\
&= f(1^2+1^2+1^2 + 3^2+3^2+3^2) = 3+3f(3)^2 \\
&= f(1^2+1^2+2^2+2^2+2^2+4^2) = 2+3f(2)^2+f(4)^2, \\
f(41) 
&= f(1^2+1^2+1^2+1^2+1^2+6^2) = 41 \\
&= f(1^2+1^2+1^2+2^2+3^2+5^2) = 3+f(2)^2+f(3)^2+f(5)^2, \\
f(21) 
&= f(1^2+1^2+1^2+1^2+1^2+4^2) = 5+f(4)^2 \\
&= f(1^2+2^2+2^2+2^2+2^2+2^2) = 1+5f(2)^2
\end{align*}
we deduce that
\[
f(2) = \pm2, \quad
f(3) = \pm3, \quad
f(4) = \pm4, \quad
f(5) = 5.
\]

Since $f(17)=17$, $f(2\cdot17)=2\cdot17$, $f(3\cdot17)=3\cdot17$, and $f(4\cdot17)=4\cdot17$ by Lemma \ref{Dubouis}, we obtain that $f(n) = n$ for $n \le 5$. If $n \ge 6$, $n(n-1) > 19$ and thus we can show $f(n)=n$.

\bigskip

\noindent iii) Assume that $k = 7$. By Lemma \ref{Dubouis} every positive integer can be represented by the sum of seven nonvanishing squares except for
\[
1, 2, 3, 4, 5, 6, 8, 9, 11, 12, 14, 17, 20.
\]
We have $f(1)=1$ and $f(7)=7$. Note that
\begin{align*}
f(31) 
&= f(1^2+1^2+1^2+1^2 + 3^2+3^2+3^2) = 4+3f(3)^2 \\
&= f(1^2+1^2+1^2 + 2^2+2^2+2^2 + 4^2) = 3+3f(2)^2+f(4)^2 \\
&= f(1^2+1^2+1^2+1^2 + 1^2+1^2+5^2) = 6+f(5)^2, \\
f(42) 
&= f(1^2+1^2+1^2+1^2+1^2+1^2+2^2 3^2) = 6+f(2)^2f(3)^2 \\
&= f(1^2 + 2^2+2^2+2^2+2^2 + 3^2 + 4^2) = 1+4f(2)^2 + f(3)^2 + f(4)^2 \\
&= f(1^2+1^2+1^2+1^2 + 2^2+3^2+5^2) = 4+f(2)^2+f(3)^2+f(5)^2, \\
f(55) 
&= f(1^2+1^2+1^2+1^2+1^2+1^2 + 7^2) = 55 \\
&= f(1^2+1^2+1^2+1^2+1^2 + 5^2+5^2) = 5+2f(5)^2.
\end{align*}

Thus, we can find that
\[
f(2) = \pm2, \quad
f(3) = \pm3, \quad
f(4) = \pm4, \quad
f(5) = \pm5.
\]
Note that $n(n-1) > 20$ for $n \ge 6$. We can conclude that $f(n)=n$ for every positive integer by the similar way.

\bigskip

\noindent iv) Assume that $k \ge 8$. By Lemma \ref{Dubouis} every positive integer can be represented by the sum of $k$ nonvanishing squares except for
\[
1, 2, \dots, k-1, k+1, k+2, k+4, k+5, k+7, k+10, k+13.
\]
Since $f(1)=1$, $f(k)=k$, and
\begin{align*}
f\!\left(k(k-1)\right) 
&= f(k) f(k-1) = k f(k-1)\\
&= f\!\left(k-1+(k-1)^2\right)= \underbrace{f(1)^2 + \dots + f(1)^2}_{k-1\text{ summands}} + f(k-1)^2,
\end{align*}
the value of $f(k-1)$ is either $1$ or $k-1$.

Since
\begin{align*}
40 
&= 1^2+1^2+1^2+1^2+6^2 \\
&= 2^2+3^2+3^2+3^2+3^2,
\end{align*}
we obtain that 
\[
(k-5) f(1)^2 + 4 f(1)^2 + \left(f(2)f(3)\right)^2 = (k-5) f(1)^2+f(2)^2+4f(3)^2
\]
or
\[
4 + \left(f(2)f(3)\right)^2 = f(2)^2+4f(3)^2.
\]

Also, from
\begin{align*}
32 
&= 1^2+1^2+1^2+1^2+1^2+3^2+3^2+3^2 \\
&= 2^2+2^2+2^2+2^2+2^2+2^2+2^2+2^2,
\end{align*}
we obtain that $5+3f(3)^2 = 8f(2)^2$.

By solving the system of these two equations we can find the following solutions:
\[
f(2)=\pm1, f(3)=\pm1
\qquad
\text{or}
\qquad
f(2)=\pm2, f(3)=\pm3.
\]

Now, let $3a+8b = 2k-1$ with some nonnegative integers $a$ and $b$. This is possible unless $2k-1 = 1,2,4,5,7,10,13$. That is, if $2k-1 > 13$, then it is represented in such a way. The number $13$ is called the \emph{Frobenius number} of $3$ and $8$. Note that 
\begin{align*}
a+b < \frac{2k-1}{3} + \frac{2k-1}{8} = \frac{11}{24}(2k-1) < k-1
\intertext{and}
(k-1) + (2k-1) + (k-1)^2 = k^2+k-1.
\end{align*}
Thus, since $2^2a+3^2b = (a+b)+(3a+8b) = (a+b)+(2k-1)$,
\begin{align*}
f(k^2 + k-1) 
&= f(k)^2 + (k-1)f(1)^2 = k^2+k-1 \\
&= f{\big(}
\hspace{-2ex}\underbrace{1^2 + \dots + 1^2}_{k-a-b-1\text{ summands}} 
\hspace{-2ex}+ \underbrace{2^2 + \dots + 2^2}_{a\text{ summands}} + \underbrace{3^2 + \dots + 3^2}_{b\text{ summands}}
+(k-1)^2{\big)} \\
&= (k-a-b-1) + af(2)^2 + bf(3)^2 + f(k-1)^2.
\end{align*}
Therefore
\[
f(2) = \pm2, \quad f(3) = \pm3, \quad f(k-1)=k-1.
\]

Since
\begin{align*}
4^2 + 2^2+2^2+2^2 &= 1^2+3^2+3^2+3^2,
&5^2 + 1^2+1^2 &= 3^2 + 3^2 + 3^2, \\
6 &= 2 \cdot 3, 
&7^2 + 1^2 &= 5^2 + 5^2, \\
8^2 + 1^2 &= 4^2 + 7^2,
&9^2 + 2^2 &= 6^2 + 7^2, \\
10 &= 2 \cdot 5,
\end{align*}
we obtain that $f(n) = \pm n$ for $n \le 10$.

Note that
\begin{align*}
(2\ell+1)^2+(\ell-2)^2 &= (2\ell-1)^2 + (\ell+2)^2 \text{ with } \ell \ge 5
\end{align*}
and
\begin{align*}
(2\ell)^2 + (\ell-5)^2 &= (2\ell-4)^2 + (\ell+3)^2 \text{ with } \ell \ge 6.
\end{align*}
Hence, by induction, $f(n) = \pm n$ for every positive integer $n$.

Thus, by Lemma \ref{Dubouis}, $f(n)=n$ for all $n > k+13$. For each $1 < n \le k+13$, let $m = n(k+13)+1$, we have $f(m)=m$ and $f(nm) = nm$ since $nm > m > k+13$. It follows from $(n,m)=1$ that $f(n)=n$.

\end{proof}

\section*{Acknowledgments}

The author would like to thank referees for their kind and valuable comments.

\end{document}